\documentclass[11pt]{amsart}
\usepackage[T1]{fontenc}
\usepackage{lmodern}
\usepackage{microtype}
\usepackage{amssymb}
\usepackage{amsthm}
\usepackage{amscd}
\usepackage[sort,numbers]{natbib}
\usepackage{hyperref}

\newtheorem{theorem}{Theorem}[section]
\newtheorem{lemma}[theorem]{Lemma}

\newtheorem{cor}[theorem]{Corollary}
\newtheorem{question}[theorem]{Question}

\theoremstyle{definition}
\newtheorem{definition}[theorem]{Definition}

\newtheorem{remark}[theorem]{Remark}

\def \dom{\operatorname{dom}}

\makeatletter

\def\dotminussym#1#2{%
  \setbox0=\hbox{$\m@th#1-$}%
  \kern.5\wd0%
  \hbox to 0pt{\hss\hbox{$\m@th#1-$}\hss}%
  \raise.6\ht0\hbox to 0pt{\hss$\m@th#1.$\hss}%
  \kern.5\wd0}

\mathchardef\mhyphen="2D

\allowdisplaybreaks[2]

\newcommand{\RCA}{\ensuremath{\mathbf{RCA_0}}}

\newcommand{\ACA}{\ensuremath{\mathbf{ACA_0}}}

\newcommand{\Pmax}{\ensuremath{\Pi^1_1\mhyphen\mathrm{MAX}}}

\begin{document}

\title{On Maximum Conservative Extensions}
\author{Henry Towsner}
\date{\today}
\thanks{Partially supported by NSF grant DMS-1157580.}
\address {Department of Mathematics, University of Pennsylvania, 209 South 33rd Street, Philadelphia, PA 19104-6395, USA}
\email{htowsner@math.upenn.edu}
\urladdr{www.math.upenn.edu/~htowsner}

\begin{abstract}
We investigate the set of $\Pi^1_2$ sentences which are $\Pi^1_1$ conservative over theories of second order arithmetic.  We exhibit new elements of these sets and conclude that the sets are $\Pi_2$ complete.  Along the way, we show that, over the theory \RCA{}, induction for $\Sigma_n$ formulas has essentially no consequences for $\Delta_{n+1}$ formulas.
\end{abstract}

\maketitle

\section{Introduction}

Many questions in reverse mathematics amount to asking about the implications among $\Pi^1_2$ sentences over a fixed background theory.  In particular, this includes the questions which most naturally line up with questions about computability: standard models of \RCA{} are essentially the Turing ideals, so asking whether one $\Pi^1_2$ sentence implies a second is closely related (but not identical!) to asking whether closure of a Turing ideal under solutions to one problem implies closure under solutions to a second.  The gap between the proof-theoretic approach embodied in reverse mathematics and the recursion theoretic approach in terms of Turing ideals is the given by the possibility of nonstandard models of arithmetic.  Therefore part of the investigation of a $\Pi^1_2$ sentence is the investigation of its purely first-order consequences.  One way to address this is to ask what $\Pi^1_1$ theory $T$ is needed so that the sentence is $\Pi^1_1$-conservative over $T$.

Yokoyama \cite{MR2798907}  and, independently, Ikeda and Tsuboi \cite{ikeda:conservativity}, showed that when $T$ is a $\Pi^1_2$ theory, there is a largest $\Pi^1_2$ theory $\Pmax(T)$ which is $\Pi^1_1$ conservative over $T$.  In the particular case where $T$ is \RCA, three (families of) members of $\Pmax(\RCA)$ are known: weak K\"onig's lemma (due to Harrington), a version of the Baire Category Theorem \cite{MR1233924}, and the existence of cohesive sets \cite{cholak:MR1825173}.  Yokoyama asked whether these statements provided axioms for $\Pmax(\RCA)$.

For extensions of Peano arithmetic, $\Pmax(T)$ is always $\Pi_2$ complete \cite{MR2014250}, but the proof is proof theoretic and does not apply to weaker theories.  In this paper we exhibit new elements of the theories $\Pmax(\RCA+I\Sigma_n)$ which suffice to show that these theories are $\Pi_2$ complete.  In particular, this means the theories are not given by a finite number of axiom schemes.

Our main tool is showing that given an arbitrary model of $\RCA+I\Sigma_n$ and an arbitrary set $X$ of elements, we may add sets to the model to obtain a model of $\RCA+I\Sigma_n$ in which $X$ is encoded by a $\Delta_{n+1}$ formula.  This shows that $\RCA+I\Sigma_n$ places essentially no restraint on $\Delta_{n+1}$ formulas, a result which may be of independent interest.

The author is grateful to Peter Cholak and Keita Yokoyama for helpful discussions on this topic, and to Fran{\c{c}}ois Dorai for pointing out a different perspective on this paper in his blog post \cite{dorais_blog}.

\section{Definitions}

Throughout this paper, we consider theories in the language $\mathcal{L}^2$ of second-order arithmetic; all theories we consider will extend the standard base theory {\RCA} (see \cite{simpson99}).  All models will have the form $\mathfrak{M}=(M,\mathcal{M})$ where $M$ is a model of first-order arithmetic and $\mathcal{M}\subseteq\mathcal{P}(|M|)$.  (As the example suggests, we will write the Fraktur letter $\mathfrak{M}$ for the model, the Roman letter $M$ for the first-order part, and the calligraphic letter $\mathcal{M}$ for the second-order part.)  We write $|M|$ for the universe of $M$ and $||M||$ for the cardinality of $|M|$.

\begin{definition}
  We say $\mathfrak{M}$ is \emph{countable} if $||M||=|\mathcal{M}|=\aleph_0$.  We say $\mathfrak{M}$ is an \emph{$\omega$-submodel} of $\mathfrak{M}'$ if $M=M'$ and $\mathcal{M}\subseteq\mathcal{M}'$.
\end{definition}

We follow the convention of using lower case letters to refer to elements of $|M|$ or numeric variables, and upper case letters for elements of $\mathcal{M}$ or set variables.  In particular, when we write $\forall x\phi$, we mean $x$ is a numeric variable, while $\forall X\phi$ indicates that $X$ is a set variable.

\begin{definition}
  Let $T\subseteq T'$ be theories in $\mathcal{L}^2$.  We say $T'$ is \emph{$\Pi^1_1$ conservative} over $T$ if whenever $T'\vdash\sigma$ and $\sigma$ is a $\Pi^1_1$ sentence, already $T\vdash\sigma$.

We say a sentence $\sigma$ is \emph{true} if $(\mathbb{N},\mathcal{P}(\mathbb{N}))\vDash\sigma$.  A theory $T$ is true if every sentence in $T$ is true.
\end{definition}

The most common way to show that some theory is $\Pi^1_1$ conservative is to show that it has the $\omega$-extension property:
\begin{lemma}\label{ext_property}
  Suppose that every countable model of $T$ is an $\omega$-submodel of a model of $T'$.  Then $T'$ is $\Pi^1_1$ conservative over $T$.
\end{lemma}
% \begin{proof}
%   Suppose $\sigma$ is a $\Pi^1_1$ sentence with $T\not\vdash\sigma$.  We may assume $\sigma$ has the form $\forall X\phi(X)$ where $\phi$ is arithmetic.  Then $T+\neg\sigma$ is consistent, so there is a countable model $\mathfrak{M}\vDash T+\neg\sigma$, and so there must be an $S\in\mathcal{M}$ such that $\mathfrak{M}\vDash\neg\phi(S)$.  $\mathfrak{M}$ is an $\omega$-submodel of some $\mathfrak{M}'\vDash T'$.  Since $M'=M$ and $S\in\mathcal{M}'$, we have $\mathfrak{M}'\vDash\neg\phi(S_1,\ldots,S_n)$, and so $\mathfrak{M}'\vDash\neg\sigma$, and therefore $T'\not\vdash\sigma$.
% \end{proof}

Cholak, Jockusch, and Slaman asked if, for $\Pi^1_2$ theories, this is the only way for a theory to be $\Pi^1_1$ conservative \cite{cholak:MR1825173}.  That is:
\begin{question}
  Suppose $T\subseteq T'$ are $\Pi^1_2$ theories extending \RCA, and that $T'$ is a $\Pi^1_1$ conservative extension of $T_0$.  Is it the case that every model of $T$ is an $\omega$-submodel of a model of $T'$?
\end{question}

Avigad showed that the answer is no \cite{avigad:collection}; this was extended by Yokoyama \cite{MR2798907}, who showed that the answer is still no even if we require that $T'$ be true.  We will provide some additional examples below.

However Yokoyama \cite{MR2798907} and Ikeda and Tsuboi \cite{ikeda:conservativity} showed that, nonetheless, two distinct $\Pi^1_2$, $\Pi^1_1$ conservative extensions of $T$ do have common models.  More precisely:
\begin{theorem}
  Let $T_0,T_1,T_2$ be $\Pi^1_2$ theories such that $T_1$ and $T_2$ are $\Pi^1_1$ conservative extensions of $T_0$.  Then $T_1+T_2$ is a $\Pi^1_1$ conservative extension of $T_0$.
\end{theorem}
By standard compactness arguments, if $T_0\subseteq T_1\subseteq\cdots$ and each $T_i$ is $\Pi^1_1$ conservative over $T_0$, so is $\bigcup_i T_i$.

\begin{definition}
  $\Pmax(T)$ is the collection of $\Pi^1_2$ sentences $\sigma$ such that $T+\sigma$ is $\Pi^1_1$ conservative over $T$.
\end{definition}

\begin{cor}
  If $T$ is a $\Pi^1_2$ theory then $\Pmax(T)$ is $\Pi^1_1$ conservative over $T$.
\end{cor}

\section{Sets at Arm's Length}
We work with some fixed internal bijective pairing function, $(\cdot,\cdot):|M|^2\rightarrow|M|$.  For longer tuples we write $(\ldots,y,z)$ as an abbreviation for $((\ldots,y),z)$.  When functions are understood to take tuples as their input, we will write $f(x,y,\ldots,z)$ instead of $f((x,y,\ldots,z))$.

\begin{definition}
  If $S\in\mathcal{S}$, we say $S$ is \emph{convergent} if for each $i\in |M|$ there is an $m\in |M|$ such that for all $m'\geq m$, $(i,m')\in S$ iff $(i,m)\in S$.  If $S$ is convergent, we write
\[\lim S=\{i\mid \mathfrak{M}\vDash\exists m\forall m'\geq m\ (i,m')\in S\}=\{i\mid\mathfrak{M}\vDash\forall m\exists m'\geq m\ (i,m')\in S\}.\]

We define $n$-convergence and $\lim^n S$ recursively.  Every $S$ is $0$-convergent and $\lim^0S=S$.  If $S$ is $n$-convergent, we say $S$ is $n+1$-convergent if $\lim^nS$ is convergent, and define $\lim^{n+1}S=\lim\lim^nS$.
\end{definition}
Of particular importance is the fact that if $S\in\mathcal{M}$, the set $\lim^nS$ is $\Delta_{n+1}$ in $\mathfrak{M}$ using $S$ as a parameter.

To actually build models, it is convenient to use the following standard result:
\begin{theorem}[Friedman \cite{friedman:restricted_induction}]\label{finish_model}
If $\mathfrak{M}$ is a model of $I\Sigma_n$ then it is an $\omega$-submodel of a model of $\RCA+I\Sigma_n$.  
\end{theorem}

Our main tool for proving results about $\Pmax(\RCA+I\Sigma_n)$ is the following theorem:
\begin{theorem}\label{main_tool_n}
  Let $\mathfrak{M}=(M,\mathcal{M})$ be a countable model of $\RCA+I\Sigma_n$, and let $W\subseteq |M|$.  Then $\mathfrak{M}$ is an $\omega$-submodel of some $\mathfrak{M}'=(M,\mathcal{M}')$ such that $\mathfrak{M}'\vDash\RCA+I\Sigma_n$ and there is some $S\in\mathcal{M}'$ with $\lim^nS=W$.
\end{theorem}

In other words, while we cannot expect to add arbitrary sets to a nonstandard model (since such a set could easily violate induction), we can add descriptions of arbitrary sets as long as we ``keep them at arm's length''---as long as the set is described only as a limit which is too complicated to be covered by the induction axioms in the underlying theory.

\subsection{The $I\Sigma_1$ Case}
We first prove the case where $n=1$ to illustrate the method.  The main tool is a forcing argument based on the technique from \cite{cholak:MR1825173}.

\begin{theorem}
  Let $\mathfrak{M}=(M,\mathcal{M})$ be a countable model of $\RCA+I\Sigma_1$, and let $W\subseteq |M|$.  Then $\mathfrak{M}$ is an $\omega$-submodel of some $\mathfrak{M}'=(M,\mathcal{M}')$ such that $\mathfrak{M}'\vDash\RCA+I\Sigma_1$ and there is some $S\in\mathcal{M}'$ with $\lim S=W$.
\end{theorem}
\begin{proof}
  By Theorem \ref{finish_model}, it suffices to show that there is an $S$ such that $\lim S=W$ and for every $\Sigma_1$ formula $\phi(x,X)$ with parameters from $\mathfrak{M}$, induction holds for $\phi(x,S)$.

We prove this by forcing.  Our conditions will be tuples $p=(\chi^p,I^p,V^p)$ where:
\begin{itemize}
\item $\chi_B^p$ is a $\{0,1\}$-valued function with bounded domain whose graph is in $\mathcal{M}$,
\item $I^p$ is finite,
\item $V^p:I^p\rightarrow |M|$,
\item Whenever $i\in I^p$, there is an $s\geq V^p(i)$ such that $\chi^p(i,s)$ is defined,
\item If $i\in I^p$, $s,s'\geq V^p(i)$ and $\chi^p(i,s)$ and $\chi^p(i,s')$ are both defined then $\chi^p(i,s)=\chi^p(i,s')$.
\end{itemize}
We define $\nu^p$ to be the partial function such that:
\begin{itemize}
\item $\chi^p\subseteq\nu^p$,
\item If $i\in I^p$ and $s\geq V^p(i)$ then $\nu^p(i,s)=\chi^p(i,s')$ where $s'\geq V^p(i)$ is such that $\chi^p(i,s')$ is defined.
\end{itemize}

We say $p'\preceq p$ if $\chi_B^p\subseteq\chi_B^{p'}$, $I^p\subseteq I^{p'}$, and $V^{p'}{\upharpoonright} I^p=V^p$.  Note that this implies $\nu^p\subseteq \nu^{p'}$.  We write $p'\preceq_0 p$ if $p'\preceq p$ and $I^p=I^{p'}$.

$\chi^p$ is an approximation to the characteristic function of the set $S$ we ultimately built.  $I^p$ is a collection of points where we have fixed the value of the limit, and $V^p$ is the point by which the limit has converged.

We consider formulas in the language $\mathcal{L}^2(G)$, given by expanding $\mathcal{L}^2$ with a fresh set constant $G$ and constants for all parameters from $\mathfrak{M}$.  It is convenient to restrict consideration to prenex formulas.

If $\phi$ is a $\Delta_0$ formula in $\mathcal{L}^2(G)$, we say $p\Vdash\phi(G)$ if for every $S$ such that $\chi_S\upharpoonright\dom(\nu^p)=\nu^p$, $\phi(S)$ holds.  We say $p\Vdash\exists x\phi(x,G)$ if whenever $p'\preceq p$, there is some $p''\preceq p'$ and some $m\in |M|$ such that $p''\Vdash\phi(m,G)$.  We say $p\Vdash\forall x\phi(x,X)$ if for every $m\in |M|$, $p'\Vdash\phi(m)$.  Note that when $\phi$ is $\Delta_0$, $p\Vdash\phi$ is a $\Delta_1$ statement.

We will construct an infinite sequence of conditions $p_0\succeq p_1\succeq\cdots$ satisfying the following properties:
\begin{itemize}
\item If $p_j\Vdash\exists x\phi(x)$ where $\phi$ is $\Delta_0$ then there is a $j'\in\mathbb{N}$ and a $k\in|M|$ such that $p_{j'}\Vdash\phi(k)$,
\item For each $\Sigma_1$ formula $\phi(x)$ in $\mathcal{L}^2(G)$, either there is a $j$ such that $p_j\Vdash\forall x\neg\phi(x)$ or there is a $j$ and a $c\in |M|$ such that $p_j\Vdash\phi(c)\wedge\forall x<c\,\neg\phi(x)$
\item For each $i\in|M|$, there is some $j$ with $i\in I^{p_j}$ and $\chi^{p_j}(i,V^{p_j}(i))=\chi_W(i)$.
\end{itemize}
We may break each of these properties up into countably many requirements.  In most forcing arguments, we would show that the sequence satisfies all these properties by showing that each requirement is satisfied by a dense set of conditions.  In order to avoid conflicts with the third property, however, we need slightly more: we need to show that the first two properties are dense with respect to $\preceq_0$, and only in order to satisfy the third property will we add elements to $I^p$.

The first set of requirements is easily handled: if $p\Vdash\exists x\phi(x)$ then there is a $q\preceq p$ and a $k$ such that $q\Vdash\phi(k)$.  Since forcing a $\Delta_0$ statement is $\Delta_1$, this depends on only boundedly many values of $\nu^q$, so we may choose a $q'\preceq_0 p$ so that $\chi^{q'}$ agrees with $\nu^q$ on a large enough bounded subset so that $q'\Vdash\phi(k)$.

To satisfy the second family of requirements, let a condition $p$ and some $\Sigma_1$ formula $\exists y\phi(x,y)$ be given.  Consider the set of $c$ such that there is some $q\preceq_0 p$ and some $m$ such that $q\Vdash\phi(c,m)$.  If there is no such $c$ then we have $p\Vdash\forall x\forall y\neg\phi(x,y)$, and we are finished.  Otherwise, the set of such $c$ is a $\Sigma_1$ set, so (since the ground model satisfies $I\Sigma_1$) there is a least such $c$, and so we may pass to some $q\preceq_0 p$ so that $q\Vdash\exists y\phi(c,y)$.  It suffices to show that if $c'<c$, $q\Vdash\forall y\neg\phi(c',y)$; suppose not, so there is some $c'<c$, some $q'\preceq q$, and some $m$ such that $q'\Vdash\phi(c',m)$.  Then, since forcing a $\Delta_0$ formula is a $\Delta_1$ property, again there is a bounded portion of $\nu^{q'}$ witnessing this, and therefore there is a $q_0\preceq_0 p$ so that $q_0\Vdash\phi(c',m)$.  But this contradicts the leastness of $c$.

The third set of requirements is satisfied by setting $I^q=I^p\cup\{i\}$, setting $V^q(i)$ to be larger than any element $s$ with $\chi^p(i,s)$ defined, and setting $\chi^q(i,V^q(i))=\chi_W(i)$.

We take $S=\{x\mid \exists i\ \nu^{p_i}(x){=}1\}$.  We have that whenever $p_i\Vdash\phi(G)$ where $\phi(G)$ is $\Pi_1$ or $\Sigma_1$, $\phi(X)$ holds: when $\phi$ is $\Delta_0$ or $\Pi_1$ this is immediate from the definition, and for $\Sigma_1$ formulas this is enforced by the first set of requirements.  The second set of requirements ensures that $\Sigma_1$ induction holds for formulas including $S$ as a parameter.  The third set of requirements ensures that $\lim S=W$.
\end{proof}

\subsection{The General Case}

We now turn to the full proof of Theorem \ref{main_tool_n}.  Throughout this argument, we work with a fixed value $n$.  To manage notation, we introduce the following shorthand: we write $\vec s_{[k,i]}$ for the finite sequence $s_k,\ldots,s_i$.  In particular, instead of writing $f(s_k,\ldots,s_i)$, we will write $f(\vec s_{[k,i]})$.

Our forcing conditions will be tuples 
\[p=(\chi^p,\{I^p_i\}_{i\leq n}, \{V^p_i\}_{i\leq n})\]
where:
\begin{itemize}
\item $\chi^p$ is a $\{0,1\}$-valued function with bounded domain whose graph is in $\mathcal{M}$,
\item $I_n^p$ is a finite set,
\item For each $i$ with $0<i<n$, $I_i^p$ is a bounded set of sequences of length $n-i+1$ which is in $\mathcal{M}$,
\item For each $i$ with $0<i\leq n$, $V_i^p:I_i^p\rightarrow |M|$ is a function in $\mathcal{M}$,
\item (\emph{Realization}) If $1<i\leq n$ and $\vec s_{[n,i]}\in I_{i}^p$ then there is some $s_{i-1}\geq V_{i}^p(\vec s_{[n,i]})$ so that $\vec s_{[n,i-1]}\in I_{i-1}^p$,
\item (\emph{Realization}) If $\vec s_{[n,1]}\in I_1^p$ then there is some $s_{0}\geq V_1^p(\vec s_{[n,1]})$ so that $\chi^p(\vec s_{[n,0]})$ is defined,
\item (\emph{Coherence}) If $m< n$, $\chi^p(\vec s_{[n,m+1]},\vec t_{[m,0]})$ and $\chi^p(\vec s_{[n,m+1]},\vec t'_{[m,0]})$ are both defined, and for every $i\in[0,m]$, $(\vec s_{[n,m+1]},\vec t_{[m,i+1]})\in I^p_{i+1}$, $t_{i}\geq V_{i+1}^p(\vec s_{[n,m+1]},\vec t_{[m,i+1]})$, and $(\vec s_{[n,m+1]},\vec t'_{[m,i+1]})\in I^p_{i+1}$, $t'_{i}\geq V_{i+1}^p(\vec s_{[n,m+1]},\vec t'_{[m,i+1]})$, then
\[\chi^p(\vec s_{[n,m+1]},\vec t_{[m,0]})=\chi^p(\vec s_{[n,m+1]},\vec t'_{[m,0]}).\]
\end{itemize}

We say $q\preceq p$ if:
\begin{itemize}
\item $\chi^p\subseteq\chi^q$,
\item For each $i$ with $0<i\leq n$, $I_i^p\subseteq I_i^q$ and $V_i^p\subseteq V_i^q$.
\end{itemize}

The only reason that that $I^p_n$ has to be handled separately from $I^p_i$ for $i<n$ is that $I^p_n$ is required to be actually finite, while $I^p_i$ is merely $\mathfrak{M}$-finite.

 The coherence property requires some unpacking to understand.  We start with $I_1^p$: we intend for $S$ to be convergent, which means that for each $s$, there should be some $v$ such that $\chi_S(s,m)$ is the same for all $m\geq v$.  When $s=\vec s_{[n,1]}\in I_1^p$, this means we have decided that $v=V_1^p(\vec s_{[n,1]})$.  Similarly, when $n=2$, we also need $\lim S$ to be convergent, which means that for each $s$ there should be some $s_1$ so that when $t_1\geq s_1$, $(\lim S)(s,t_1)$ is the same for all $t_1\geq s_1$, which in turn means that $\lim_{m\rightarrow\infty}\chi_S(s,t_1,m)$ converges to the same value for each $t_1\geq s_1$ (though the point of convergence may itself depend on the choice of $t_1$).

To make this property easier to work with, we introduce some terminology.  We say that a sequence $(s_n,\ldots,s_0)$ is \emph{convergent from $m$ in $p$} if for every $i\in[0,m]$, $\vec s_{[n,i+1]}\in I^p_{i+1}$ and $s_i\geq V_{i+1}^p(\vec s_{[n,i+1]})$.  The purpose of the realization property is to ensure that whenever $\vec s_{[n,m+1]}\in I_{m+1}^p$, there is some sequence $\vec s_{[n,0]}$ which is convergent from $m$ in $p$.  Coherence then states that whenever $(\vec s_{[n,m+1]},\vec t_{[m,0]})$ and $(\vec s_{[n,m+1]},\vec t'_{[m,0]})$ are both convergent from $m$ and $\chi^p$ is defined on both, $\chi^p$ has the same value on both.  Therefore whenever $\vec s_{[n,m+1]}\in I^p_{m+1}$, we define $\chi_*^p(\vec s_{[n,m+1]})$ to be this uniquely defined value.

Note that whenever $\vec s_{[n,0]}$ is convergent from $m$ in $p$, $\vec s_{[n,0]}$ is also convergent from every $m'\leq m$, and so $\chi^*_p(\vec s_{[n,m+1]})=\chi^*_p(\vec s_{[n,m'+1]})$.

For any $i$ with $0\leq i\leq n$, we say $q\preceq_i p$ if $q\preceq p$ and for all $j>i$, $I^p_j=I^q_j$ (and therefore $V^p_j=V^q_j$).  ($q\preceq_n p$ is simply $q\preceq p$.)  When $q\preceq p$ we write $q\upharpoonright_i p$ for the condition with $\chi^{q\upharpoonright_i p}=\chi^q$ and $I_j^{q\upharpoonright_i p}=I_j^q$ if $j\leq i$ and $I_j^{q\upharpoonright_i p}=I_j^p$ if $j>i$.  It is easy to check that $q\upharpoonright_i p$ is a condition and $q\upharpoonright_i p\preceq_i p$.

If $q,r\preceq p$, $q\preceq_i p$ and $r\upharpoonright_i p=p$, we define $q\cup r$ to be the condition with $\chi^{q\cup r}=\chi^q$, $I_j^{q\cup r}=I_j^q$ if $j\leq i$ and $I_j^{q\cup r}=I_i^r$ if $i<j$.  Note that $q\cup r$ need \emph{not} be a condition itself: for instance, when $i=0$, $r$ could include some new element $(\ldots,s)$ in $I_1^r$ with $V^r_1(\ldots,s)=v$, while $\chi^q$ could be newly defined on $(\ldots,s,v+1)$ with a different value from $\chi^p(\ldots,s,v)$.  However the realization property sharply limits how $r$ could expand and still satisfy $r\upharpoonright_{i+1}p=p$, so we have the following:
\begin{lemma}\label{cond_union}
  If $q\preceq_i p$ and $r\upharpoonright_{i+1}p=p$ then $q\cup r$ is a condition.
\end{lemma}
\begin{proof}
  We need only check the coherence property of $q\cup r$.  Suppose $(\vec s_{[n,m+1]},\vec t_{[m,0]})$ and $(\vec s_{[n,m+1]},\vec t'_{[m,0]})$ are both convergent from $m<n$ and $\chi^q$ is defined on both.  If $m\leq i$ then since the property of being convergent from $m$ depends only on $I^{q\cup r}_j, V^{q\cup r}_j$ for $j\leq m+1$, both sequences were convergent from $m$ in $q$, and therefore by the coherence of $q$, $\chi^{q\cup r}=\chi^q$ takes the same value on both.

So consider the case where $m>i$.  Then since $(\vec s_{[n,m+1]},\vec t_{[m,0]})$ is convergent from $m$, $(\vec s_{[n,m+1]},\vec t_{[m,0]})$ must also be convergent from $i<m$, and $(\vec s_{[n,m+1]},\vec t_{[m,i+1]})$ must have already belonged to $I^p_{i+1}$, and so
\[\chi^r_*(\vec s_{[n,m+1]})=\chi^r_*(\vec s_{[n,m+1]},\vec t_{[m,i+1]})=\chi^p_*(\vec s_{[n,m+1]},\vec t_{[m,i+1]}).\]
But also $\chi^q_*(\vec s_{[n,m+1]},\vec t_{[m,i+1]})=\chi^p_*(\vec s_{[n,m+1]},\vec t_{[m,i+1]})$, so in particular $\chi^q(\vec s_{[n,m+1]},\vec t_{[m,0]})=\chi^r_*(\vec s_{[n,m+1]})$.

By a similar argument, $\chi^q(\vec s_{[n,m+1]},\vec t'_{[m,0]})=\chi^r_*(\vec s_{[n,m+1]})$, so we have
\[\chi^q(\vec s_{[n,m+1]},\vec t_{[m,0]})=\chi^q(\vec s_{[n,m+1]},\vec t'_{[m,0]})\]
as desired.
\end{proof}

 We define $\nu^p$ to be a $\{0,1\}$-valued partial function in $\mathcal{M}$ such that:
 \begin{itemize}
 \item $\chi^p\subseteq\nu^p$,
 \item If $\vec s_{[n,1]}\in I^p_1$ then for any $s_0\geq V^p_1(\vec s_{[n,1]})$, $\nu^p(\vec s_{[n,0]})=\chi^p_*(\vec s_{[n,1]})$.
 \end{itemize}
Clearly if $q\preceq p$ then $\nu^p\subseteq\nu^q$.

We define when a condition forces a formula.  We again consider formulas from the language $\mathcal{L}(G)$, which extends $\mathcal{L}^2$ by constants for all elements of $\mathfrak{M}$ and a fresh set constant $G$.  It will suffice to only consider formulas in prenex form.
\begin{itemize}
\item If $\phi$ is $\Delta_0$, $p\Vdash\phi(G)$ if whenever $\chi_S\upharpoonright\dom(\nu^p)=\nu^p$, $\phi(S)$ holds,
\item $p\Vdash\exists x\phi(x)$ if $\forall q\preceq p\,\exists r\,\preceq q\,\exists k\ r\Vdash\phi(k)$,
\item $p\Vdash\forall x\phi(x)$ if $\forall k\ p\Vdash\phi(k)$.
\end{itemize}
When $\phi$ is in prenex form, we will abuse notation and write $\neg\phi$ for the prenex form of the negation.

It is easy to see that when $\phi$ is $\Delta_0$, the set of $p$ such that $p\Vdash\phi$ is $\Delta_1$.  We need to extend this so that the set of $p$ forcing $\Sigma_m$ or $\Pi_m$ sentences, for $m<n$, is defined internally.

\begin{lemma}
  If $q\preceq p$ and $p\Vdash\phi$ then $q\Vdash\phi$.
\end{lemma}
\begin{proof}
  By induction on the complexity of $\phi$.  If $\phi$ is $\Delta_0$, this is immediate from the definition.  If $\phi$ is $\Pi_{m+1}$, this follows directly from the inductive hypothesis.  If $\phi$ is a $\Sigma_{m+1}$ formula $\exists x\psi(x)$, for any $q'\preceq q$, also $q'\preceq p$, so there is an $r\preceq q'$ and a $k$ such that $r\Vdash\psi(k)$, as needed.
\end{proof}

\begin{lemma}\label{neg_flip}
  If $p\not\Vdash\phi$ then there is a $q\preceq p$ such that $q\Vdash\neg\phi$.
\end{lemma}
\begin{proof}
  By induction on $\phi$.  This is immediate when $\phi$ is $\Delta_0$.  If $p\not\Vdash\exists x\psi(x)$ then there is a $q\preceq p$ such that for every $k$ and every $r\preceq q$, $r\not\Vdash\psi(k)$.  Therefore by the inductive hypothesis, $q\Vdash\neg\psi(k)$ for every $k$, and therefore $q\Vdash\forall m\neg\psi(k)$, as desired.

If $p\not\Vdash\forall x\psi(x)$ then there is a $k$ such that $p\not\Vdash\psi(k)$, and so by the inductive hypothesis, there must be some $q\preceq p$ such that $q\Vdash\neg\psi(k)$, so $q\Vdash\exists x\neg\psi(k)$.
\end{proof}

\begin{lemma}
Suppose $m<n$, $q\preceq p$ and $q\Vdash\phi$ where $\phi$ is a $\Sigma_{m+1}$ formula.  Then there is a $q'\preceq_m p$ so that $q'\Vdash\phi$.
\end{lemma}
\begin{proof}
By induction on $m$.  Suppose $m=0$ and $q\Vdash\exists x\phi(x)$; we may fix an $r\preceq q$ and a $k$ so that $r\Vdash\phi(k)$.  Since forcing a $\Delta_0$ formula depends on boundedly many values of $\nu^r$, we may find a $q'\preceq_0 p$ by taking $\chi^{q'}$ to agree with $\nu^r$ on these values, and therefore $q'\Vdash\phi(k)$, so $q'\Vdash\exists x\phi(x)$.

For $m>0$, suppose $q\Vdash\exists x\forall y\phi(x,y)$ where $\phi$ is a $\Sigma_{m-1}$ formula.  We may fix an $r\preceq q$ and a $k$ so that $r\Vdash\forall y\phi(k,y)$.  Consider $r\upharpoonright_m p$; if $r\upharpoonright p\Vdash\forall y\phi(k,y)$ then we are done since $(r\upharpoonright_m p)\preceq_m p$.  Otherwise, by the previous lemma, there must be some $r'\preceq(r\upharpoonright_m p)$ such that $r'\Vdash\exists y\neg\phi(k,y)$, and by the inductive hypothesis, some such $q'\preceq_{m-1}(r\upharpoonright_m p)$.  

Let $p'=r\upharpoonright_m p$.  Then $q'\preceq_{m-1}p'$ and $r\upharpoonright p'=p'$, so Lemma \ref{cond_union} applies, and $q'\cup r$ is a condition.  Since $q'\cup r\preceq q'$, we have $q'\cup r\Vdash\exists y\neg\phi(k,y)$.  Since $q'\cup r\preceq r$, we have $q'\cup r\Vdash\forall y\phi(k,y)$.  This is a contradiction, so $r\upharpoonright_m p\Vdash\forall y\phi(k,y)$.
\end{proof}
This implies that when $q\preceq p$ and $q\Vdash\phi$ where $\phi$ is a $\Pi_m$ formula, there is a $q'\preceq_m p$ so that $q'\Vdash\phi$.

Let $\exists x\psi(x)$ be a $\Sigma_{m}$ formula with $m<n$.  By Lemma \ref{neg_flip}, $p\Vdash\exists x\psi(x)$ iff there is no $q\preceq p$ such that $q\Vdash\forall x\neg\psi(x)$, and this in turn is equivalent to there being no $q'\preceq_m p$ such that $q'\Vdash\forall x\neg\psi(x)$.  Note that quantification over $q'\preceq_m p$ is internal, so by induction on the complexity of formulas, when $m<n$ and $\exists x\phi(x)$ is a $\Sigma_m$ formula, the set of $p$ such that $p\Vdash\exists x\phi(x)$ and the set of $p$ such that $p\Vdash\forall x\exists y\phi(x,y)$ are both $\Pi_{m+1}$.

\begin{lemma}[$I\Sigma_n$]\label{lemma_induction}
For any $p$ and any $\Sigma_n$ formula $\exists y\phi(x,y)$, there is a $q\preceq_{n-1} p$ such that one of the following holds:
\begin{itemize}
\item $q\Vdash\forall x\forall y\neg\phi(x,y)$,
\item There are a $c$ and a $k$ such that $q\Vdash\phi(c,k)$ and for every $c'<c$, $q\Vdash\forall y\neg\phi(c',y)$.
\end{itemize}
\end{lemma}
\begin{proof}
Let $p$ and $\phi$ be given.  Consider the set of $c$ such that there exists a $q\preceq_{n-1}p$ and a $k$ with $q\Vdash\phi(c,k)$.

Suppose there is no such $c$.  If $q\preceq p$ and $q\Vdash\phi(c,k)$ then, since $\phi(c,k)$ is a $\Pi_{n-1}$ formula, there would be some $q'\preceq_{n-1}p$ such that $q'\Vdash\phi(c,k)$.  So there is no such $q\preceq p$, and therefore $p\Vdash\neg\phi(c,k)$.  Since this holds for every $c,k$, $p\Vdash\forall c\forall k\neg\phi(c,k)$.

Suppose there is such a $c$.  Since the set of $q$ such that $q\Vdash\phi(c,k)$ is given by a $\Pi_{n-1}$ formula, the existence of such a $q,k$ is a $\Sigma_n$ formula, so there must be a least $c$ such that there exist such a $q,k$.  We choose some $q$ and some $k$ so that $q\Vdash\phi(c,k)$, and so $q\Vdash\exists x\phi(c,k)$.  Suppose $q'\preceq q$ and there is a $c'< c$ such that $q'\Vdash\exists y\phi(c',y)$; then there would be an $r\preceq q'$ and a $k$ such that $r\Vdash\phi(c',k)$, contradicting the leastness of $c$.  So there is no such $q'$, and so for every $c'<c$, each $q'\preceq q$ must have $q'\not\Vdash\exists y\phi(c',y)$, so $q\Vdash\forall y\neg\phi(c',y)$.
\end{proof}

\begin{proof}[Proof of Theorem \ref{main_tool_n}]
Again, by Theorem \ref{finish_model}, it suffices to show that there is an $S$ such that $\lim^n S=W$ and for every $\Sigma_n$ formula $\phi(x,G)$ with parameters from $\mathfrak{M}$, the least number principle holds for $\{x\mid \phi(x,S)\}$.  We construct an infinite sequence $p_1\succ p_2\succ\cdots$ such that:
\begin{itemize}
\item If $\exists x\phi(x)$ is a $\Sigma_m$ formula with $m\leq n$ and $p_i\Vdash\exists x\phi(x)$ then there is a $j$ and a $k$ such that $p_j\Vdash\phi(k)$,
\item For every $\Sigma_n$ formula $\exists y\phi(x,y)$, either there is a $j$ with $p_j\Vdash\forall x\forall y\neg\phi(x,y)$ or a $c,k,i$ with $p_i\Vdash\phi(c,k)$,
\item For each $r<n$, $\bigcup_i I^{p_i}_r=|M|$,
\item $\bigcup_i I^{p_i}_n=|M|$ and when $s_n\in I^{p_i}_n$ then $\chi_*^{p_i}(s_n)=\chi_W(s_n)$.
\end{itemize}
Each of these families consists of countably many requirements, so it suffices to show that given any $p$ and any individual requirement, there is a $q\preceq p$ satisfying the additional requirement.  In order to avoid conflict with the fourth part, when satisfying the first three families, we must have $q\preceq_{n-1} p$.

For the first family, observe that if $p\Vdash\exists x\phi(x)$ then there must, by definition, be a $q\preceq p$ and a $k$ with $q\Vdash\phi(k)$, and since $\phi(k)$ is a $\Pi_{m-1}$ formula with $m-1<n$, we may take $q\preceq_{n-1}p$.  The existence of $q\preceq_{n-1}p$ satisfying the second family is given by the preceding lemma.  The existence of $q\preceq_{n-1}p$ satisfying the third family is trivial

For the final condition, we may take any $s_n\not\in I^p_n$ and define $q$ by $I^q_n=I^p_n\cup\{s_n\}$, adding realizations larger than any elements where $\chi^p$ is defined, and setting $\chi^q$ on the resulting sequence equal to $\chi_W(s_n)$.

Suppose we have such a sequence.  We take $S=\{k\mid\exists i\ \nu^{p_i}(k)=1\}$.  We claim that if $\phi(G)$ is a $\Sigma_n$ or $\Pi_n$ formula and there is an $i$ with $p_i\Vdash\phi(G)$ such that $\phi(S)$ holds.  We show this by induction on $\phi$; when $\phi$ is $\Delta_0$, this is immediate from the definition.  Suppose $p_i\Vdash\exists x\phi(x,G)$; by the construction, there is a $j\geq i$ and a $k$ so that $p_j\Vdash\phi(k,G)$, and so by IH, $\phi(k,S)$, and so $\exists x\phi(x,S)$.  If $p_i\Vdash\forall x\phi(x,G)$ then for each $k$, $p_i\Vdash\phi(k,G)$, so by IH, $\phi(k,S)$ holds, and therefore $\forall x\phi(x,S)$ holds.

Together with the second requirement, this implies that the least number principle holds for $\{x\mid\phi(x,S)\}$ whenever $\phi$ is a $\Sigma_n$ formula.

It remains to show that $\lim^nS=W$.  We show by induction on $r$ that $\lim^rS$ exists and for each $\vec s_{[n,r]}$, $\chi_{\lim^r S}(\vec s_{[n,r]})=\chi_*^{p_i}(\vec s_{[n,r]})$ for some (and therefore every) $i$ large enough that $\vec s_{[n,r]}\in I^{p_i}_r$.  When $r=1$, whenever $\vec s_{[n,1]}\in I^{p_i}_1$, we have $\nu^{p_i}(\vec s_{[n,0]})=\chi_*^{p_i}(\vec s_{[n,1]})$ for coboundedly many values of $s_0$.  Since for any $\vec s_{[n,1]}$ there is some $p_i$ with $\vec s_{[n,1]}\in I^{p_i}_1$, the claim follows immediately.

Consider some $r>1$ and some $\vec s_{[n,r]}\in I^{p_i}_r$; for every $s_{r-1}\geq V^{p_i}_r(\vec s_{[n,r]})$, there must be some $j$ such that $\vec s_{[n,r]}\in I^{p_j}_{r-1}$, and therefore by the inductive hypothesis, $\chi_{\lim^{r-1}X}(\vec s_{[n,r-1]})$ exists and is equal to $\chi_*^{p_j}(\vec s_{[n,r-1]})=\chi_*^{p_i}(\vec s_{[n,r]})$.  Since this holds for every $\vec s_{[n,r]}$, and, given some $\vec s_{[n,r]}$, for coboundedly many $s_{r-1}$, the claim follows.

In particular, it follows from the fourth requirement that $\lim^nS$ exists and is equal to $W$.
\end{proof}

\section{Conservation over Fragments of Arithmetic}

The results in the previous section make it easy to prove a number of results about $\Pmax(\RCA+I\Sigma_n)$.

\begin{definition}
  $I\subseteq M$ is a \emph{cut} if $I$ is an initial segment of $M$ closed under successor.
\end{definition}

\begin{theorem}
  The set of G\"odel numbers of formulas in $\Pmax(\RCA+I\Sigma_n)$ is $\Pi_2$ complete.
\end{theorem}
\begin{proof}
  It is easy to see that $\Pmax(\RCA+I\Sigma_n)$ is $\Pi_2$: $\sigma\in\Pmax(\RCA+I\Sigma_n)$ exactly if $\sigma$ is $\Pi^1_2$ and every for proof of a $\Pi^1_1$ formula in $\RCA+I\Sigma_n+\sigma$, there is a proof of that formula in $\RCA+I\Sigma_n$.

  On the other hand, for any $\Pi_2$ sentence $\forall x\exists y\phi(x,y)$ with $\phi$ a $\Delta_0$ formula, consider the $\Sigma^1_1$ sentence $\Phi$:
  \begin{quote}
There exists a set $X$ such that $\lim^nX$ is a cut and for any $x\in\lim^nX$, there is a $y\in \lim^nX$ such that $\phi(x,y)$ holds.
  \end{quote}

If $\mathbb{N}\vDash\forall x\exists y\phi(x,y)$ then whenever $\mathfrak{M}$ is a countable model of $\RCA+I\Sigma_n$, we may apply Theorem \ref{main_tool_n} to obtain a set extension $\mathfrak{M}'$ such that $\mathbb{N}=\lim^nS$ for some $S\in\mathcal{M}'$.  Therefore $\mathfrak{M}'\vDash\Phi$.  Since this holds for any countable model, by Lemma \ref{ext_property} it follows that $\RCA+I\Sigma_n+\Phi$ is a $\Pi^1_1$ conservative extension of $\RCA+I\Sigma_n$.

On the other hand, if $\mathbb{N}\not\vDash\forall x\exists y\phi(x,y)$, there must be some $n$ such that $\mathbb{N}\vDash\forall y\neg\phi(n,y)$.  However $\RCA+I\Sigma_n+\Phi$ implies $\exists y\phi(n,y)$; since $\mathbb{N}\vDash\RCA+I\Sigma_n$, it follows that $\RCA+I\Sigma_n\not\vdash\exists y\phi(n,y)$, so $\exists y\phi(n,y)$ is an arithmetic formula implied by $\RCA+I\Sigma_n+\Phi$ but not by $\RCA+I\Sigma_n$.
\end{proof}

\begin{remark}\ 
  \begin{itemize}
  \item This immediately answers Question 3.5 of \cite{MR2798907}: $\Pmax(\RCA)$ is not computably axiomatizable.
%  \item The sentences $\Phi$ which belong to $\Pmax(\RCA+I\Sigma_n)$ were true, so the subset of $\Pmax(\RCA+I\Sigma_n)$ consisting of true formulas is $\Pi_2$ hard.
  \item If $T$ is any theory which is not $\Pi^1_1$ conservative over $\RCA+I\Sigma_n$ then we may take any $\Pi^1_1$ formula $\theta$ with $T\vdash\theta$ but $\RCA+I\Sigma_n\not\vdash\theta$, and we could carry out the argument above using the formulas $\Phi\vee\theta$ (one can manipulate the quantifiers to get this into a prenex $\Pi^1_2$ form).  So the situation is also not simplified in the restriction to sentences provable in $T$.  This answers Question 3.2 of \cite{MR2798907}.
  \item In particular, if $\theta$ is a true $\Pi^1_1$ formula not implied by $\RCA+I\Sigma_n$, the argument could be carried out using sentences $\Phi\vee\theta$, which are true.
    % \item ******Extend to theories whose axioms have bounded quantifier complexity
  \end{itemize}
\end{remark}

A more complicated use of this method, together with an idea from \cite{MR0409172}, can answer Question 3.3 of \cite{MR2798907}:
\begin{theorem}
  There is a $\Pi^1_2$ theory $T$ which is a $\Pi^1_1$ conservative extension of $\RCA+I\Sigma_n$ and a nonstandard model $\mathfrak{M}$ which is not an $\omega$-submodel of any model of $T$.
\end{theorem}
\begin{proof}
Recall that a model $M$ is \emph{recursively saturated} if whenever $\langle\lceil\phi_i(x,\vec y)\rceil\rangle$ is a computable sequence of G\"odel numbers of (arithmetic) formulas and $\vec a\in M^{|\vec y|}$ is a tuple such that, for each $n$, $M\vDash\exists x\forall i\leq n\phi_i(x,\vec a)$, then there is an $m\in |M|$ such that $\forall i\, M\vDash \phi_i(m,\vec a)$.  (It does not change the notion if we replace recursive with $\Sigma_1$, an idea which will help us below.)  See, for instance, \cite{MR0403952}, for general properties of recursively saturated models.  We fix, as usual, some uniformly computable enumeration $\varphi_e$ of the partial computable functions.  We write $\varphi_e(m)\downarrow$ to indicate that the computation of $\varphi_e(m)$ converges.

Consider the formula $\mathrm{Sat}(k,X)$ which states that $X$ codes the satisfiability relation for formulas with G\"odel number below $k$.  That is, for $\lceil\psi\rceil<k$, $\langle \lceil\psi\rceil,\vec a\rangle\in X$ iff the usual Tarski conditions hold (if $\psi$ is atomic then only when $\psi(\vec a)$ holds, if $\psi$ is $\exists x\theta$ then only if there is some $m$ such that $\langle \lceil\theta\rceil,\vec a^\frown m\rangle\in X$, and so on).  This is an arithmetic formula.  Write $X\upharpoonright k=\{(\lceil\psi\rceil,\vec a)\in X\mid \lceil\psi\rceil<k\}$.

%If $I$ is a cut, let us say a sequence is \emph{$I$-bounded computable} if there is an index $e\in I$ such that, for each $m\in I$, the computation $\{e\}(m)$ halts in $s$ steps for some $s\in I$.

We will encode a cut, $W_0$, and a satisfaction relation $W_1$ which satisfies $\mathrm{Sat}(k,W_1\upharpoonright k)$ for every $k\in W_0$.  We then wish to consider types relative to this cut: that is, we take an index $e\in W_0$ and view it as an attempting at defining a type by taking $T_e$ to consist of those formulas $\psi$ such that $\varphi_e(m)=\lceil\psi\rceil$ for some $m$.  We can think think of a type as being ``satisfiable'' below some $m$ if there is an $x$ such that whenever $m'<m$ and $\varphi_e(m')\downarrow$ and $\varphi_e(m')\in W_0$, $\langle \varphi_e(m'),x\rangle\in W_1$.  (Our precise definition below is slightly more complicated, since it allows for parameters.)  Recursive saturation is similar to having an induction property for this notion: either there is an $m$ so that the type is satisfiable below $m$ but not $m+1$, or the type is fully satisfiable---there is a single $x$ with $\langle\varphi_e(m),x\rangle\in W_1$ whenever $\varphi_e(m)\downarrow$ and $\varphi_e(m)\in W_0$.

We now make this precise.  We consider the sentence $\Phi$ given by:
\begin{quote}
  There is a set $S=(S_0,S_1)$ such that, taking $W_0=\lim^n S_0$ and $W_1=\lim^n S_1$:
  \begin{enumerate}
  \item $W_0$ is a cut,
  \item For $k\in W_0$, $\mathrm{Sat}(k,W_1\upharpoonright k)$,
  \item \label{cond:key} Let a tuple $\vec a$ and an element $e\in W_0$ be given; then either:
    \begin{itemize}
    \item There is an $m\in W_0$ such that $\varphi_e(m)\downarrow$, $\varphi_e(m)\in W_0$, but $\varphi_e(m)$ is not the G\"odel number of a formula $\psi$ with $|\vec a|+1$ free variables,
    \item There is an $m\in W_0$ and an $x$ such that for every $m'<m$ such that $\varphi_e(m')\downarrow$ and $\varphi_e(m')\in W_0$, $\langle\varphi_e(m'),(x,\vec a)\rangle\in W_1$, but there is no $x'$ such that for every $m'\leq m$ such that $\varphi_e(m')\downarrow$ and $\varphi_e(m')\in W_0$, $\langle\varphi_e(m),(x',\vec a)\rangle\in W_1$, or
    \item There is an $x$ such that for every $m\in W_0$ such that $\varphi_e(m)\downarrow$ and $\varphi_e(m)\in W_0$, $\langle\varphi_e(m),(x,\vec a)\rangle\in W_1$.
    \end{itemize}
  \end{enumerate}
\end{quote}
Let $\mathfrak{M}$ be recursively saturated; then if $W_0=\mathbb{N}$ and $W_1$ is the actual satisfiability relation for $M$, the pair $(S_0,S_1)$ will satisfy $\Phi$.  Therefore by Theorem \ref{main_tool_n}, $\mathfrak{M}$ is an $\omega$-submodel of a model of $\Phi$.

Conversely, suppose $\mathfrak{M}\vDash\Phi$, let $\langle\lceil\phi_i(x,\vec y)\rceil\rangle=\langle\varphi_e(i)\rangle$ be an actual computable sequence of formulas (that is, $e\in\mathbb{N}$ and $\varphi_e$ is total computable) and let $\vec a$ be such that for each $n\in\mathbb{N}$, $M\vDash\exists x\forall i\leq n\phi_i(x,\vec a)$.  The definition of satisfiability is unique for actual formulas, and since $\varphi_e(i)$ converges to an element in $\mathbb{N}\subseteq W_0$ for all $i\in\mathbb{N}$, for each $n\in\mathbb{N}$ there is an $x$ so that for every $i\leq n$, $\langle \varphi_e(i),(x,\vec a)\rangle\in W_1$.  Consider the cases given by condition (\ref{cond:key}); in the second case, let $x$ and $m$ be given so that for every $m'<m$ such that $\varphi_e(m')\downarrow$, $\langle\varphi_e(m'),(x,\vec a)\rangle\in W_1$.  We cannot have $m\in\mathbb{N}$ since the sequence is finitely satisfiable, so $m>\mathbb{N}$.  But then $\langle\varphi_e(n),(x,\vec a)\rangle\in W_1$ for all $n$.

In the third case, it is immediate that $x$ is the desired witness.
\end{proof}
This method could also be used to produce a $\Pi^1_2$ sentence $\Phi'$ so that $\mathfrak{M}$ can be expanded to satisfy $\Phi'$ iff $(M,S_1,\ldots,S_n)$ is recursively saturated for any choice of finitely many $S_1,\ldots,S_n\in\mathcal{M}$.

\section{Conservation over \ACA}

The method we used to show that $\Pmax(\RCA+I\Sigma_n)$ is $\Pi_2$ complete depended on the fact that in theories like $\RCA+I\Sigma_n$ there is a ``mismatch'' between the restrictions on possible sets placed by the axioms and our ability to describe sets using arithmetic formulas---that is, since the axioms all have low quantifier complexity, we were able to encode arbitrary sets by making sure they could only be decoded using formulas more complicated than the axioms.  Over \ACA, the comparable result has been shown by proof theoretic methods \cite{MR744645,MR567671,quinsey_thesis}.  For completeness, and to compare with the work above, we include a proof here.

The following argument from \cite{avigad:collection} provides inspiration:
\begin{theorem}
  There is a $\Sigma^1_1$ sentence false in the standard model which is $\Pi^1_1$ conservative over $\ACA$.
\end{theorem}
\begin{proof}
  A sentence $\sigma$, obtained using the fixed point lemma, is:
\[    \exists \phi\exists X \left[\phi(X)\wedge\ACA+\exists X\phi(X)\vdash\neg\sigma)\right].\]
That is, $\sigma$ says:
\begin{quote}
  My negation is provable in $\ACA$ from a true $\Sigma^1_1$ sentence.
\end{quote}
(We use here the fact that in \ACA{} there is a $\Sigma^1_1$ truth predicate for $\Sigma^1_1$ sentences.)

This sentence is clearly false in the standard model, since if it were true, its negation would actually be provable, leading to a contradiction.

To see that $\sigma$ is $\Pi^1_1$ conservative, suppose $\ACA+\sigma\vdash\forall X\neg\phi(X)$.  Then $\ACA+\exists X\phi(X)\vdash\neg\sigma$.  But, working in $\ACA+\exists X\phi(X)$, we see that $\exists X\phi(X)$ is true and that $\ACA+\exists X\phi(X)\vdash\neg\sigma$, so $\ACA+\exists X\phi(X)\vdash\sigma$ as well.  This is a contradiction, so $\ACA\vdash\forall X\neg\phi(X)$, as desired.
\end{proof}

\begin{theorem}
  The set of G\"odel numbers of formulas in $\Pmax(\ACA)$ is $\Pi_2$ complete.
\end{theorem}
\begin{proof}
We write $T\vdash_p\phi$ for the statement that $p$ is the G\"odel number of a proof of $\phi$ from the set of formulas $T$.

  Given a quantifier free formula $\psi$, consider the following sentence $\sigma_\psi$ obtained using the fixed point lemma:
\[    \exists \phi\exists X \left[\phi(X)\wedge\forall y(\exists z\psi(y,z)\vee\exists p\leq y\ \ACA+\exists X\phi(X)\vdash_p\neg\sigma_\psi)\right].\]
This sentence says:
\begin{quote}
Either $\forall y\exists z\psi(y,z)$ or the first counterexample gives a bound for a proof of my negation from $\ACA$ plus a true $\Sigma^1_1$ sentence.
\end{quote}

Suppose $\mathbb{N}\vDash\forall y\exists z\psi(y,z)$ and that
\[\ACA+\sigma_\psi\vdash\forall X\neg\phi(X).\]
Then also
\[\ACA+\exists X\phi(X)\vdash\neg\sigma_\psi.\]
In particular, there is a natural number $p$ which is the G\"odel number of this proof, and so, together with the finitely many witnesses that when $y<p$ there is a $z$ with $\psi(y,z)$, we obtain a finite number witnessing that
\[\forall y(\exists z\psi(y,z)\vee\exists p\leq y \ACA+\exists X\phi(X)\vdash_p\neg\sigma_\psi)\]
and therefore $\ACA$ proves this.  But then $\ACA+\exists X\phi(X)\vdash\sigma_\psi$.  Since $\ACA+\exists X\phi(X)\vdash\sigma_\psi\wedge\neg\sigma_\psi$, we have
\[\ACA\vdash\forall X\neg\phi(X),\]
showing that $\ACA+\sigma_\psi$ is $\Pi^1_1$ conservative over \ACA.

On the other hand, suppose $\mathbb{N}\vDash\exists y\forall z\neg\psi(y,z)$.  Fix the least $n$ such that $\mathbb{N}\vDash\forall z\neg\psi(n,z)$.  If there is no $p\leq n$ such that $\ACA+\exists X\phi(X)\vdash_p\neg\sigma_\psi$ for some $\phi$ then $\ACA+\sigma_\psi\vdash\exists z\psi(n,z)$, an arithmetic statement which is false, and so not a consequence of \ACA.  Suppose there is such a $p$; there are finitely many $p_1,\ldots,p_k$ below $n$ such that for each $i$ there is a $\phi_i$ so that $\ACA+\exists X\phi_i(X)\vdash_{p_i}\neg\sigma_\psi$.  Therefore $\ACA+\sigma_\psi+\forall z\neg\psi(n,z)$ must prove $\exists X\phi_1(X)\vee\cdots\exists X\phi_k(X)$.  It follows that $\ACA+\sigma_\psi+\forall z\neg\psi(n,z)\vdash\neg\sigma_\psi$.  This is a contradiction, so again $\ACA+\sigma_\psi\vdash\exists z\psi(n,z)$.
\end{proof}
The same argument applies to any computable axiomatized extension of \ACA.

\section{Conclusion}

We do not see any way to adapt the latter proof to the case of $\RCA+I\Sigma_n$, nor to adapt proof for $\RCA+I\Sigma_n$ to the case of $\ACA$.  It would be interesting to know if there is a single proof which can apply to both cases.

We have not considered other fragments of $\ACA$ which have been studied in the literature.  The most important other family of fragments are those given by the bounded collection scheme, $B\Sigma_n$.
\begin{question}
  Is $\Pmax(\RCA+B\Sigma_n)$ $\Pi_2$ complete?
\end{question}
The same question could be asked of other natural families of fragments of arithmetic, so we also wonder what the limits to this result are.
\begin{question}
  Is there a computably axiomatized fragment $T$ of Peano arithmetic so that $\Pmax(\RCA+T)$ is computable?  Is there such a fragment with $\Pmax(\RCA+T)$ computably enumerable?  Is there such a fragment where the axioms have bounded quantifier complexity?
\end{question}

Over $\ACA$ there is a known theory of $\Pi^1_3$, $\Pi^1_2$ conservative extensions (see \cite{simpson99}), some of it developed using techniques which inspired ideas in this paper (especially \cite{MR0409172,MR0403952}).  Analogously to $\Pmax$, there are largest $\Pi^1_{n+1}$, $\Pi^1_n$ conservative theories for all $n\geq 1$, and these theories are, by the same argument as for $n=1$, $\Pi_2$.
\begin{question}
  Is the largest $\Pi^1_{n+1}$, $\Pi^1_n$ conservative extension of $T$ $\Pi_2$ complete, for $n>2$ and $T$ either $\RCA+I\Sigma_m$ or $\ACA$?
\end{question}

\bibliographystyle{plain}
\bibliography{../../Bibliographies/main}
\end{document}